\documentclass[a4paper,10pt]{amsart}

\usepackage{amsaddr}
\usepackage{amsmath}
\usepackage{amssymb}
\usepackage{amsthm}
\usepackage{enumitem}
\usepackage[T1]{fontenc}
\usepackage{mathtools}
\usepackage{microtype}
\usepackage{mathrsfs}
\usepackage[dvipsnames]{xcolor}
\usepackage[all,cmtip]{xy}

\usepackage[pagebackref]{hyperref}
\usepackage{amsrefs}
\usepackage[capitalise]{cleveref}

\hypersetup{
	colorlinks=true,
	linkcolor=MidnightBlue,
	citecolor=MidnightBlue,
	urlcolor=MidnightBlue
}

\theoremstyle{plain}
\newtheorem{theorem}{Theorem}[section]

\newtheorem{conjecture}[theorem]{Conjecture}
\newtheorem{corollary}[theorem]{Corollary}

\newtheorem{lemma}[theorem]{Lemma}
\newtheorem{proposition}[theorem]{Proposition}

\theoremstyle{definition}
\newtheorem{definition}[theorem]{Definition}

\newtheorem*{ack}{Acknowledgement}

\theoremstyle{remark}

\newtheorem{remark}[theorem]{Remark}

\newtheorem*{remark*}{Remark}
\newtheorem*{notation*}{Notation and Terminology}

\numberwithin{equation}{section}

\DeclareMathOperator{\Aut}{Aut}

\DeclareMathOperator{\GL}{GL}

\DeclareMathOperator{\Mat}{Mat}

\DeclareMathOperator{\SL}{SL}

\DeclareMathOperator{\id}{id}

\DeclarePairedDelimiter{\abs}{\lvert}{\rvert}

\DeclareMathOperator{\Bim}{Bim}

\DeclareMathOperator{\Fix}{Fix}

\DeclarePairedDelimiterX{\pair}[2]{\langle}{\rangle}{#1,#2}

\DeclareMathOperator{\NS}{NS}

\newcommand{\Romannum}[1]{\uppercase\expandafter{\romannumeral #1}}
\newcommand{\romannum}[1]{\romannumeral #1\relax}

\title[Automorphism Groups]
{Automorphism Groups of Compact Complex Surfaces: T-Jordan Property, Tits Alternative and Solvability}

\author{Jia~Jia}
\address{National University of Singapore, Singapore 119076, Republic of Singapore}
\email{jia\_jia@u.nus.edu}


\begin{document}

\begin{abstract}
	Let \(X\) be a (smooth) compact complex surface.
	We show that every torsion subgroup of the biholomorphic automorphism group \(\Aut(X)\)
	is virtually nilpotent.
	Moreover, we study the Tits alternative of \(\Aut(X)\) and virtual derived length of virtually solvable subgroups of \(\Aut(X)\).
\end{abstract}

\maketitle
\setcounter{tocdepth}{1}
\tableofcontents

\section{Introduction}\label{sec:introduction}

\subsection{T-Jordan Property}\label{sub:t_jordan_property}

The below theorem is our starting point.

\begin{theorem}[\cite{lee1976torsion}]\label{thm:connected_lie_group_t_jordan}
	Let \(G\) be a connected Lie group.
	Then there is a constant \(J = J(G)\) such that every torsion subgroup of \(G\)
	contains a (normal) abelian subgroup of index \(\leq J\).
\end{theorem}

Inspired by the \emph{Jordan property} introduced by Popov \cite{popov2011makarlimanov}*{Definition~2.1},
we propose the following generalisation.

\begin{definition}
	A group \(G\) is called \emph{T-Jordan} (alternatively, we say that \(G\) has the \emph{T-Jordan} property)
	if there is a constant \(J(G)\) such that every torsion subgroup \(H\) of \(G\) has an abelian subgroup \(H_1\)
	with the index \([H:H_1]\leq J(G)\).
\end{definition}

For a compact complex manifold \(X\), we denote by \(\Aut(X)\) the group of biholomorphic automorphisms.
As in \cite{popov2011makarlimanov}, we propose the following conjecture.

\begin{conjecture}\label{conj:main_t_jordan}
	Let \(X\) be a compact complex manifold.
	Then \(\Aut(X)\) is T-Jordan.
\end{conjecture}

Using the equivariant K\"ahler model,
Meng and the author gave an affirmative answer to \cref{conj:main_t_jordan}
for compact complex manifolds in Fujiki's class \(\mathcal{C}\).
A compact complex space is in \emph{Fujiki's class \(\mathcal{C}\)} if
it is the meromorphic image of a compact K\"ahler manifold
(cf.~\cite{fujiki1978automorphism}*{Definition~1.1}).

However,
since we only use \cref{thm:class_c} below in dimension two
and the classical result of Fujiki \cite{fujiki1978automorphism}*{Theorem~4.8}
and Lieberman \cite{lieberman1978compactness}*{Proposition~2.2} is enough for its proof.

\begin{theorem}[cf.~{\cite{jia2022equivariant}*{Corollary~1.5}}]\label{thm:class_c}
	Let \(X\) be a compact complex space which is in Fujiki's class \(\mathcal{C}\).
	Then \(\Aut(X)\) is T-Jordan.
\end{theorem}

In this paper, we are focusing on smooth compact complex surfaces,
especially the non-K\"ahler (\(=\) non-Fujiki's class \(\mathcal{C}\)) surfaces.
We use the following notation.

\begin{itemize}[wide=0pt,leftmargin=*]
	\item Let \(\Xi\) be the set of smooth compact complex surface \(X\) in class \Romannum{7}
	      with the algebraic dimension \(a(X) = 0\) and the second Betti number \(b_2(X) > 0\).
	\item Let \(\Xi_0 \subseteq \Xi\) be those minimal surfaces which have no curve.
\end{itemize}
Here, surfaces of class \Romannum{7} are those smooth compact complex surfaces with
the first Betti number \(b_1 = 1\) and Kodaira dimension \(\kappa = -\infty\) (cf.~\cite{kodaira1964structure}*{\S~7}).
By a \emph{minimal surface} we mean a smooth compact complex surface that does not contain any \((-1)\)-curve \(C\)
(i.e., a smooth rational curve \(C\) with self-intersection \(C^2 = -1\)).
It is known that a smooth compact complex surface \(X\) is minimal if and only if
any bimeromorphic holomorphic map \(X\longrightarrow X'\) to a smooth compact complex surface \(X'\) is an isomorphism.


\begin{proposition}\label{prop:main_t_jordan}
	Let \(X\) be a compact complex surface not in \(\Xi_0\).
	Then \(\Aut(X)\) is T-Jordan.
\end{proposition}

Although we are not able to confirm \cref{conj:main_t_jordan} for surfaces in \(\Xi_0\),
we have a slightly weaker result on torsion subgroups of \(\Aut(X)\).
Recall that a group \(G\) is \emph{virtually \(\mathcal{P}\)},
if there exists a finite-index subgroup of \(G\) that has the property \(\mathcal{P}\).
We say that a group \(G\) is \emph{\(\mathcal{P}\)-by-\(\mathcal{Q}\)}
if \(G\) fits into an short exact sequence
\[
	1 \longrightarrow K \longrightarrow G \longrightarrow H \longrightarrow 1
\]
where \(K\) has the property \(\mathcal{P}\) and \(H\) has the property \(\mathcal{Q}\).

\begin{proposition}\label{prop:main_class_vii_no_curve}
	Let \(X\) be a smooth compact complex surface in \(\Xi_0\).
	Let \(G \leq \Aut(X)\) be a torsion subgroup.
	Then \(G\) is virtually abelian.
\end{proposition}

Our main result below is a direct consequence of \cref{prop:main_t_jordan,prop:main_class_vii_no_curve}.

\begin{theorem}\label{thm:virtually_abelian}
	Let \(X\) be a smooth compact complex surface.
	Then any torsion subgroup \(G\leq \Aut(X)\) is virtually abelian.
\end{theorem}

Note that by the global spherical shell (GSS) conjecture, \(\Xi_0 = \emptyset\)
(cf.~\cite{nakamura1990surfaces}*{Conjecture~1}).

\begin{remark}
	Assume the global spherical shell conjecture.
	Let \(X\) be a smooth compact complex surface.
	Then \(\Aut(X)\) is T-Jordan.
\end{remark}

\subsection{Tits Alternative}\label{sub:tits_alternative}

In group theory, the \emph{Tits alternative} is an outstanding theorem about the structure of linear groups
(cf.~\cite{tits1972free}).
Campana, Wang and Zhang studied the automorphism groups of compact K\"ahler manifolds
and proved a Tits type alternative for such groups (cf.~\cite{campana2013automorphism}*{Theorem~1.5}).
We generalise their celebrated result to compact complex spaces in Fujiki's class \(\mathcal{C}\)
and smooth compact complex surfaces.

\begin{theorem}\label{thm:tits_alternative_class_c}
	Let \(X\) be a compact complex space in Fujiki's class \(\mathcal{C}\).
	Then \(\Aut(X)\) satisfies the Tits alternative,
	that is, for any subgroup \(G\) of \(\Aut(X)\) either it contains a non-abelian free subgroup,
	or it is virtually solvable, i.e., admits a solvable subgroup of finite index.
\end{theorem}

For the definition and properties of Enoki surfaces and Inoue-Hirzebruch surfaces,
see \cite{enoki1981surfaces} and \cite{inoue1977new}, respectively.

\begin{theorem}\label{thm:main_tits_alternative}
	Let \(X\) be a smooth compact complex surface.
	Assume that either \(X\notin \Xi\),
	or \(X \in \Xi\) but its minimal model is an Enoki or Inoue-Hirzebruch surface.
	Then \(\Aut(X)\) satisfies the Tits alternative.
\end{theorem}

\subsection{Virtual Derived Length}\label{sub:virtual_derived_length}

Once we have established the Tits type alternative for \(\Aut(X)\),
it is natural to study the properties of those virtually solvable subgroups.
%

We have the following result, which gives a uniform upper bound of the virtual derived length (cf.~\cref{def:derived_Length})
for virtually solvable subgroups of \(\Aut(X)\).

\begin{theorem}\label{thm:main_derived_length}
	Let \(X\) be a smooth compact complex surface.
	Assume that either \(X\notin \Xi\),
	or \(X \in \Xi\) but its minimal model is an Enoki or Inoue-Hirzebruch surface.
	Let \(G \leq \Aut(X)\) be a virtually solvable subgroup.
	Then the virtual derived length \(\ell_{\mathrm{vir}}(G) \leq 4\).
\end{theorem}

\begin{remark}
	\hfill
	\begin{enumerate}[wide=0pt,leftmargin=*]
		\item Currently, we are not able to prove \cref{thm:main_tits_alternative,thm:main_derived_length} in full generality for \(X \in \Xi\).
		\item The GSS conjecture claims that any minimal surface in \(\Xi\) is a Kato surface
		      (cf.~\cite{nakamura1990surfaces}*{Conjecture~1} and \cite{dloussky2003class}*{Main Theorem}).
		\item Kato surfaces are divided into four classes:
		      Enoki surfaces (including parabolic Inoue surfaces), half Inoue surfaces, Inoue-Hirzebruch surfaces and intermediate surfaces
		      (cf.~\cite{teleman2019nonkahlerian}*{\S3.3.2}).
		\item Fix \(b > 0\).
		      The moduli space of framed Enoki surfaces with \(b_2 = b\) is an open subspace
		      of the moduli space of framed Kato surfaces with \(b_2 = b\)
		      (cf.~\cite{dloussky1999classification}*{Remark, pp.~54} and \cite{teleman2019nonkahlerian}*{\S3.3.4.1}).
		\item When \(X\) is a parabolic Inoue surface (which is in \(\Xi\)),
		      it has been proved that \(\Aut(X)\) is virtually abelian
		      (cf.~\cite{fujiki2009automorphisms}*{Theorem~1.1}).
		      In particular, \(\Aut(X)\) satisfies the Tits alternative and
		      any subgroup \(G\) of \(\Aut(X)\) is virtually solvable with virtual derived length \(\ell_{\mathrm{vir}}(G) \leq 1\).
	\end{enumerate}
\end{remark}

\begin{ack}
	The author would like to thank Professors De-Qi Zhang and Sheng Meng
	for many inspiring discussions and valuable suggestions to improve the paper,
	Professor Fabio Perroni for bring \cite{enoki1981surfaces} to his attention
	and the valuable discussion,
	and Doctor Shramov for improving \cref{prop:main_class_vii_no_curve} (and hence \cref{thm:virtually_abelian}).
	The authors would also like to thank the referee for the very careful reading and suggestions to improve the paper.
	The author is supported by a President's Scholarship of NUS.
\end{ack}

\section{Preliminaries}\label{sec:preliminaries}

In this section, we gather some general preliminary results.

\begin{definition}\label{def:derived_Length}
	Recall that for a group \(G\), its \emph{\(n\)-th derived subgroup} is defined recursively by
	\[
		G^{(0)}\coloneqq G \quad \text{and} \quad G^{(n)}\coloneqq [G^{(n-1)}, G^{(n-1)}], \, n \in \mathbb{N}.
	\]
	By definition, the group \(G\) is \emph{solvable} if and only if \(G^{(n)} = {1}\) for some non-negative integer \(n\).
	The minimum of such \(n\) is the \emph{derived length} of \(G\) (when \(G\) is solvable), denoted by \(\ell(G)\).

	If \(G\) is virtually solvable, we define the \emph{virtual derived length} to be
	\[
		\ell_{\mathrm{vir}}(G) \coloneqq \min_{G'} \ell(G')
	\]
	where \(G'\) runs through all finite-index subgroups of \(G\).
\end{definition}

A group \(G\) has \emph{bounded torsion subgroups} if
there is a constant \(T = T(G)\) such that any torsion subgroup \(G_1 \leq G\) has order \(\abs{G_1} \leq T\).

\begin{lemma}\label{lem:t_jordan_exact_sequence}
	Let
	\[
		1\longrightarrow N\longrightarrow G\longrightarrow H
	\]
	be an exact sequence of groups.
	Then the following assertions hold:
	\begin{enumerate}
		\item If \(N\) is T-Jordan and \(H\) has bounded torsion subgroups,
		      then \(G\) is T-Jordan.
		\item Assume further that the sequence is also right exact.
		      If \(N\) is a torsion group and \(G\) is T-Jordan, then \(H\) is T-Jordan.
	\end{enumerate}
\end{lemma}

\begin{proof}
	(1) is clear.
	(2)
	Let \(H'\leq H\) be a torsion subgroup, and \(G'\leq G\) be its preimage.
	Since \(N\) is a torsion group,
	\(G'\) is also a torsion group and hence there exists an abelian subgroup
	\(G'_1\leq G'\) such that \([G':G'_1]\leq J(G)\) for some constant \(J(G)\).
	Let \(H'_1 = G'_1/(G'_1 \cap N)\leq H'\), which is abelian with \([H':H'_1]\leq J(G)\).
	Therefore, \(H\) is T-Jordan.
\end{proof}

The following two lemmas are adapted from \cite{prokhorov2021automorphism}*{Corollaries~4.2, 4.4}.
Note that every element of a torsion group has finite order.

\begin{lemma}\label{lem:fix_point_torsion_embedding}
	Let \(X\) be an irreducible Hausdorff reduced complex space,
	and let \(\Gamma\leq \Aut(X)\) be a torsion subgroup.
	Suppose that \(\Gamma\) has a fixed point \(x\) on \(X\).
	Then the natural representation
	\[
		\Gamma\longrightarrow \GL(T_{x,X})
	\]
	is faithful.
\end{lemma}

\begin{lemma}\label{lem:torsion_quotient_embedding}
	Let \(X\) be an irreducible Hausdorff reduced complex space,
	and let \(\Delta\leq \Aut(X)\) be a subgroup.
	Suppose that \(\Delta\) has a fixed point \(x\) on \(X\),
	and let
	\[
		\sigma\colon \Delta\longrightarrow \GL(T_{x,X})
	\]
	be the natural representation.
	Suppose that there is a normal subgroup \(\Gamma\leq \Delta\) with the quotient group \(\Delta/\Gamma\) being torsion,
	such that the restriction \(\sigma|_{\Gamma}\) is a group monomorphism.
	Then \(\sigma\) is an embedding as well.
\end{lemma}

\begin{proof}
	Let \(\Delta_0\) be the kernel of \(\sigma\).
	Since \(\Delta_0\cap \Gamma=\{\id\}\) and \(\Delta/\Gamma\) is torsion,
	we see that \(\Delta_0\) is also torsion.
	Thus, \(\Delta_0\) is trivial by~\cref{lem:fix_point_torsion_embedding}.
\end{proof}

\begin{lemma}\label{lem:derived_length_linear_gp}
	Let \(G\leq \GL_n(\mathbb{C})\) be a virtually solvable subgroup.
	Then \(\ell_{\mathrm{vir}}(G)\leq n\)
\end{lemma}

\begin{proof}
	By passing to a finite-index subgroup,
	we may assume that \(G\) itself is solvable.
	Then some subgroup \(G_1\) of finite index in \(G\) can be put in triangular form
	(cf.~\cite{borel2012linear}*{Corollary, pp.~137}).
	It follows that the derived length of \(G_1\) and hence
	the virtual derived length of \(G\) is at most \(n\).
	%
\end{proof}

\begin{lemma}\label{lem:virtually_solvable_ses}
	Consider the short exact sequence of groups
	\[
		1\longrightarrow N\longrightarrow G\longrightarrow H\longrightarrow 1.
	\]
	Then
	\begin{enumerate}[wide=0pt,leftmargin=*]
		\item if \(N\) is solvable and \(H\) is virtually solvable,
		      then \(G\) is virtually solvable with
		      \(\ell_{\mathrm{vir}}(G) \leq \ell(N) + \ell_{\mathrm{vir}}(H)\).
		\item if \(N\) is finite and \(H\) is virtually solvable,
		      then \(G\) is virtually solvable with \(\ell_{\mathrm{vir}}(G) \leq \ell_{\mathrm{vir}}(H) + 1\);
		\item the group \(G\) is virtually solvable if and only if both \(N\) and \(H\) are virtually solvable;
		\item if both \(N\) and \(H\) satisfy the Tits alternative, then so does \(G\).
	\end{enumerate}
\end{lemma}

\begin{proof}
	(1)
	Let \(H'\leq H\) be a solvable subgroup of finite index (with \(\ell(H') = \ell_{\mathrm{vir}}(H)\)),
	and let \(G'\leq G\) be its preimage.
	Since \([G:G']=[G/N:G'/N]=[H:H'] < \infty\), we may replace \(G,H\) by \(G',H'\)
	and assume that \(H\) is solvable.
	Then it is clear that \(G\) is solvable,
	and \(\ell_{\mathrm{vir}}(G)\leq \ell(G)\leq \ell(N) + \ell(H) \leq \ell(N) + \ell_{\mathrm{vir}}(H)\).

	(2)
	As in (1), we may assume that \(H\) is solvable.
	Let \(Z_G(N)\) be the centraliser of \(N\) in \(G\).
	Since \(N\) is finite, \(Z_G(N) \leq G\) is of finite index.
	Note that \(Z_G(N) \cap N\) is abelian (in particular, solvable) and \(Z_G(N)/(Z_G(N) \cap N) \leq G/N\simeq H\) is solvable.
	Hence, \(Z_G(N)\) is also solvable with derived length \(\ell(Z_G(N)) \leq \ell(H) + 1\).
	It follows that \(\ell_{\mathrm{vir}}(G) \leq \ell_{\mathrm{vir}}(H) + 1\).

	(3)
	The ``only if'' direction is clear.
	Suppose that \(N\) and \(H\) are virtually solvable.
	As in (2), we may assume that \(H\) is solvable.
	Let \(N'\lhd N\) be a finite-index solvable (normal) subgroup of minimal index.
	Then \(N'\) is normal in \(G\);
	otherwise, if \(N''\) is another conjugate of \(N'\), then \(N'N''\lhd N\) is solvable with smaller index,
	a contradiction to the choice of \(N'\).
	Consider the following short exact sequence
	\[
		1\longrightarrow N/N' \longrightarrow G/N' \longrightarrow G/N\simeq H\longrightarrow 1.
	\]
	By assumption, the first term is finite and \(H\) is solvable.
	Thus, \(G/N'\) is virtually solvable by (2).
	Then using the short exact sequence
	\[
		1\longrightarrow N' \longrightarrow G\longrightarrow G/N' \longrightarrow 1,
	\]
	we obtain that \(G\) is virtually solvable by (1).

	(4)
	This follows from (3).
	Note that if \(H\) contains a copy of \(\mathbb{Z}*\mathbb{Z}\), a non-abelian, free group with two generators,
	then \(G\) contains a(nother) copy of  \(\mathbb{Z}*\mathbb{Z}\) as well.
\end{proof}

Using the lemma above, we can give a proof of \cref{thm:tits_alternative_class_c}.

\begin{proof}[Proof of \cref{thm:tits_alternative_class_c}]
	Let \(\pi\colon X' \longrightarrow X\) be an \(\Aut(X)\)-equivariant resolution of singularities
	(cf.~\cite{bierstone1997canonical}*{Theorem~13.2(2)(ii)})
	with \(\Aut(X)\) lifts to a (unique) subgroup of \(\Aut(X')\) via \(\pi\).
	Note that \(X'\) is still in Fujiki's class \(\mathcal{C}\).
	Therefore, we may replace \(X\) by \(X'\) and assume that \(X\) is smooth.

	Denote by
	\[
		\Aut_{\tau}(X)\coloneqq \{g\in \Aut(X)\mid g^{*}|_{H^{2}(X,\mathbb{Q})} = \id\}.
	\]
	Then we have the following short exact sequence
	\[
		1\longrightarrow G_{\tau}\longrightarrow G\longrightarrow G|_{H^{2}(X,\mathbb{Q})}\longrightarrow 1
	\]
	where \(G_{\tau}=G\cap \Aut_{\tau}(X)\).
	It suffices to prove that \(G_{\tau}\) satisfies the Tits alternative
	by \cite{tits1972free}*{Theorem~1} and \cref{lem:virtually_solvable_ses}~(4).
	Let \(\sigma\colon \widetilde{X}\longrightarrow X\) be a bimeromorphic holomorphic map from a compact
	K\"ahler manifold \(\widetilde{X}\) such that \(\Aut_{\tau}(X)\) lifts to \(\widetilde{X}\)
	holomorphically via \(\sigma\) (cf.~\cite{jia2022equivariant}*{Theorem~1.1}).
	Viewing \(G_{\tau}\leq \Aut_{\tau}(X)\) as a subgroup of \(\Aut(\widetilde{X})\),
	we know that \(G_{\tau}\) satisfies the Tits alternative (cf.~\cite{campana2013automorphism}*{Theorem~1.5}).
	Therefore, \(G\) satisfies the Tits alternative.
\end{proof}

The theorem below is due to \cite{pinkham1984automorphisms}*{\S2, Theorem}.

\begin{theorem}\label{thm:inoue_hirzebruch}
	Let \(X\) be an Inoue-Hirzebruch surface.
	Then \(\Aut(X)\) is finite.
	In particular, \(\Aut(X)\) satisfies the Tits alternative
	and every subgroup \(G \leq \Aut(X)\) is finite and hence virtually solvable with
	virtual derived length \(\ell_{\mathrm{vir}}(G) = 0\).
\end{theorem}

The following lemmas will be used in \cref{sec:compact_kahler_surfaces}.

\begin{lemma}[cf.~\cite{hall1950topology}*{\S2, 4}]\label{lem:finite_index_normal}
	Let \(G\) be a finitely generated group.
	Then every finite-index subgroup of \(G\) contains a finite-index subgroup \(G'\)
	which is characteristic in \(G\).
\end{lemma}

\begin{lemma}\label{lem:one_dim_normal_subgroup}
	Let \(G\) be a solvable group and \(U \lhd G\) a non-trivial unipotent (normal) subgroup.
	Then, after replacing \(G\) by a finite-index subgroup,
	there is a subgroup \(Z_1 \simeq \mathbb{G}_a \leq U\) which is normal in \(G\).
\end{lemma}

\begin{proof}
	Let \(Z = Z(U)\) be the centre of \(U\) which is nontrivial and normal in \(G\).
	Note that the commutative unipotent group \(Z\) is isomorphic to \(\mathbb{G}_a^n\),
	the additive affine algebraic group of dimension \(n\) for some \(n \geq 1\).
	Consider the conjugation action of \(G\) on \(Z\):
	\[
		c_g\colon x \mapsto gxg^{-1} \quad \text{for } x \in Z \text{ and } g \in G.
	\]
	This conjugation action \(c_g\) is linear:
	for a fixed \(x\in Z\simeq \mathbb{G}_a^n\),
	the morphism \(z\mapsto f(z)\coloneqq c_g(z.x)-z.c_g(x)\)
	is trivial for all integer points \(z\in \mathbb{G}_a\),
	so \(f(z)=0\) for all points \(z\in \mathbb{G}_a\).
	Therefore, this induces a homomorphism \(G\longrightarrow \GL_n(\mathbb{C})\).
	Replacing \(G\) by a finite-index subgroup,
	we may assume that the image \(G'\) of the solvable group \(G\) in \(GL_n(\mathbb{C})\) is of triangular form
	(cf.~\cite{borel2012linear}*{Corollary, pp.~137}).
	Let \(Z_1 \simeq \mathbb{G}_a\) be an (\(1\)-dimensional) eigenspace of \(G'\).
	Then \(Z_1\) is stable under the conjugation action of \(G\)
	and hence is normal in \(G\).
\end{proof}

\section{Equivariant Fibrations}\label{sec:equivariant_fibrations}

In this section we will consider smooth compact complex surfaces \(X\) with algebraic dimension \(a(X) = 1\),
or with Kodaira dimension \(\kappa(X) = 1\).
If \(a(X) = 1\), there is a natural elliptic fibration,
the \emph{algebraic reduction} \(\pi\colon X \longrightarrow Y\) (cf.~\cite{barth2004compact}*{VI.~(5.1)~Proposition}),
which is a holomorphic map.
Here, by an \emph{elliptic fibration},
we mean a surjective morphism between two compact complex spaces with connected fibre,
such that a general fibre is isomorphic to a smooth elliptic curve.
For \(g \in \Aut(X)\), the image of a fibre \(F\) of \(\pi\) under \(g\) is another fibre;
otherwise the self-intersection number of \(g(F) + F\) is positive
and hence \(X\) is projective (cf.~\cite{barth2004compact}*{\Romannum{4}, Theorem~6.2}),
a contradiction (cf.~\cite{barth2004compact}*{\Romannum{4}, Corollary~6.5}).
Thus, \(\pi\) is \(\Aut(X)\)-equivariant.
When \(\kappa(X) = 1\),
the \emph{pluricanonical fibration} gives a natural elliptic fibration as well,
which is also \(\Aut(X)\)-equivariant.

In addition to the above,
we will also study affine \(\mathbb{A}^1\)-bundle \(\pi\colon A\longrightarrow E\)
over an elliptic curve \(E\).
Note that the fibre of \(\pi\colon A\longrightarrow E\) being an affine line \(\mathbb{A}^1\)
cannot dominate the base elliptic curve \(E\).
It follows that \(\pi\) is \(\Aut(A)\)-equivariant.

Let \(\pi\colon X\longrightarrow Y\) be a morphism between complex spaces.
Let \(G\) be a complex Lie group contained in \(\Aut(X)\) such that \(\pi\circ g=\pi\)
for all \(g\in G\), i.e., \(G\leq \Aut_Y(X)\).
For any finite subset \(S\subseteq Y\),
denote by \(G_S\coloneqq \{g\in G\mid g|_{\pi^{-1}(S)}=\id\}\),
i.e., \(G_S\) is the kernel of the natural group homomorphism \(G \longrightarrow \prod_{y \in S} \Aut(X_y)\).
Note that any Lie group is second countable and hence has only countably many connected components.

The ideas of \cref{lem:discrete_G,prop:very_general_subset} come from Sheng Meng.

\begin{lemma}\label{lem:discrete_G}
	Suppose that \(\dim G = 0\).
	Then \(G_y\) is trivial for very general \(y\in Y\).
\end{lemma}

\begin{proof}
	Let \(Y_g\coloneqq \{y\in Y \mid g|_{X_y}=\id\}\).
	Then \(Y_g\) is a proper closed subspace of \(Y\) for \(g\neq \id\).
	Let \(Z\coloneqq \bigcup_{\id\neq g\in G}Y_g\).
	Then any \(y\in Y\setminus Z\) is a very general choice.
	Here we use the condition that \(G\) is countable since \(\dim G = 0\).
\end{proof}

A very general finite subset \(S\subseteq Y\) means that every element of \(S\) is very generally chosen in \(Y\).

\begin{proposition}\label{prop:very_general_subset}
	For very general finite subset \(S\subseteq Y\), \(G_S\) is trivial.
\end{proposition}

\begin{proof}
	First, we show by induction on \(\dim G\) that \(\dim G_S=0\) for very general finite subset \(S\subseteq Y\).
	When \(\dim G = 0\), the claim follows from \cref{lem:discrete_G}.
	Now assume that \(\dim G \geq 1\).
	Denote by \(T\coloneqq \{y\in Y\mid G_0|_{X_y}=\id\}\),
	where \(G_0\) is the neutral component of \(G\).
	Note that \(T\subsetneq Y\) is a closed subspace.
	Then for \(y\notin T\), we have \(\dim G_y<\dim G\).
	By induction, \(\dim(G_y)_S=0\) for any \(y\notin T\) and very general finite subset \(S\subseteq Y\) (depending on \(y\)).
	Note that \((G_y)_S = G_{\{y\}\cup S}\).
	So the claim is proved.

	Then by \cref{lem:discrete_G},
	we see that \(G_{S \cup \{y\}} = (G_S)_y\) is trivial for very general \(y\in Y\).
\end{proof}

\begin{lemma}\label{lem:algebraic_reduction_base_finite}
	Let \(X\) be either a (primary or secondary) Kodaira surface, or a minimal smooth compact complex surface of Kodaira dimension \(1\).
	Then there is an \(\Aut(X)\)-equivariant elliptic fibration \(\pi\colon X\longrightarrow Y\)
	to a smooth curve \(Y\).
	Also, the natural image of \(\Aut(X)\) in \(\Aut(Y)\) is finite.
\end{lemma}

\begin{proof}
	We may take \(\pi\) to be either the algebraic reduction or the pluricanonical fibration
	(see the first paragraph of this section).
	Then our claim follows from \cite{shramov2020finite}*{Corollary~3.4, Lemma~3.5}
	and \cite{prokhorov2020bounded}*{Proposition~1.2}.
\end{proof}

\begin{theorem}\label{thm:elliptic_fibration_all}
	Let \(X\) be either a (primary or secondary) Kodaira surface or a minimal smooth compact complex surface of Kodaira dimension \(1\).
	Then \(\Aut(X)\) is virtually abelian.
	In particular, \(\Aut(X)\) satisfies T-Jordan property, the Tits alternative
	and every subgroup \(G \leq \Aut(X)\) is virtually abelian and hence virtually solvable with
	virtual derived length \(\ell_{\mathrm{vir}}(G) \leq 1\).
\end{theorem}

\begin{proof}
	By \cref{lem:algebraic_reduction_base_finite}, there is a short exact sequence
	\[
		1\longrightarrow \Aut_Y(X)\longrightarrow \Aut(X)\longrightarrow \Gamma\longrightarrow 1,
	\]
	where \(\pi\colon X\longrightarrow Y\) is an elliptic fibration and
	\(\Gamma\) is a finite subgroup of \(\Aut(Y)\).
	By \cref{prop:very_general_subset},
	there exists a finite subset \(S\) of \(Y\) such that
	the natural map \(\Aut_Y(X)\hookrightarrow \prod_{y\in S}\Aut(X_y)\) is an injection.
	Since \(X_y\) is an elliptic curve,
	\(\prod_{c\in S}\Aut(X_y)\) and hence \(\Aut_Y(X)\) are virtually abelian.
	Then it is clear that \(\Aut(X)\) is virtually abelian.
\end{proof}

\begin{theorem}\label{thm:enoki}
	Let \(X \in \Xi\) be an Enoki surface.
	Then \(\Aut(X)\) is virtually solvable (and hence \(\Aut(X)\) satisfies the Tits alternative)
	and every subgroup \(G \leq \Aut(X)\) is virtually solvable with \(\ell_{\mathrm{vir}}(G) \leq 3\).
\end{theorem}

In the proof below,
by \emph{a cycle of rational curves},
we mean a reduced (singular) complex curve with only nodal singularities such that
its dual graph is a cycle and
each component of its normalisation is a smooth rational curve.

\begin{proof}
	By \cite{enoki1981surfaces}*{Main Theorem},
	\(X\) has a cycle of rational curves \(C\)
	such that \(A \coloneqq X \setminus C\),
	the complement of \(C\) in \(X\),
	is an affine \(\mathbb{A}^1\)-bundle \(\pi\colon A \longrightarrow E\) over an elliptic curve \(E\).
	Since the algebraic dimension \(a(X)\) is \(0\),
	the surface \(X\) contains finitely many (irreducible) curves
	(cf.~\cite{barth2004compact}*{IV.\ Theorem~8.2}).
	Replacing \(\Aut(X)\) by a finite-index subgroup,
	we may assume that \(C\) is \(\Aut(X)\)-invariant.
	Then \(\Aut(X)\) acts on \(A = X \setminus C\) biholomorphically and faithfully
	and \(\pi\) is \(\Aut(X)\)-equivariant
	(see the second paragraph of this section).

	Consider the short exact sequence
	\[
		1\longrightarrow K\longrightarrow \Aut(X)\longrightarrow \Gamma\longrightarrow 1,
	\]
	where \(K \leq \Aut_{E}(A)\) and \(\Gamma\leq \Aut(E)\).
	Since \(E\) is an elliptic curve,
	\(\Aut(E)\) and hence \(\Gamma\) are virtually abelian,
	or equivalently, virtually solvable with virtual derived length \(\leq 1\).
	By \cref{prop:very_general_subset},
	there exists a finite subset \(S\) of \(E\) such that
	the natural map \(K \hookrightarrow \prod_{e \in S}\Aut(A_e)\) is an injection.
	Since \(A_e \simeq \mathbb{A}^1\), the group
	\(\prod_{e \in S}\Aut(A_e)\) and hence \(K\) are metabelian,
	or equivalently, solvable with derived length \(\leq 2\).
	By \cref{lem:virtually_solvable_ses}~(1),
	\(\Aut(X)\) is virtually solvable with virtual derived length \(\leq 3\).
	It is clear that any subgroup \(G \leq \Aut(X)\) is also virtually solvable
	and its virtual derived length \(\ell_{\mathrm{vir}}(G) \leq 3\).
\end{proof}

\section{Surfaces with Invariant Curves}\label{sec:surfaces_with_invariant_curves}

\begin{lemma}\label{lem:cpt_surface_with_inv_curves}
	Let \(X\) be a smooth compact complex surface.
	Suppose that there is a finite non-empty \(\Aut(X)\)-invariant set \(\Sigma\) of (irreducible) curves on \(X\).
	Then \(\Aut(X)\) is T-Jordan.
\end{lemma}

\begin{proof}
	Let \(C\) be one of the curves from \(\Sigma\).
	Then the group \(\Aut(X,C)\) of automorphisms of \(X\) that preserves the curve \(C\)
	has finite index in \(\Aut(X)\).
	We only need to show that \(\Aut(X,C)\) is T-Jordan.

	Assume first that \(C\) is singular.
	Then \(\Aut(X,C)\),
	after replacing by a subgroup of finite index,
	fixes some point on \(C\).
	Now \cref{lem:fix_point_torsion_embedding} implies that \(\Aut(X,C)\) is embedded into \(\GL_2(\mathbb{C})\).
	Therefore, the group \(\Aut(X,C)\) is T-Jordan (cf.~\cref{thm:connected_lie_group_t_jordan}).

	Now we assume that \(C\) is smooth and let \(G \leq \Aut(X,C)\) be a torsion subgroup.
	Let \(\mathcal{N}_{C/X}\) be the normal bundle of \(C\) in \(X\).
	Consider the natural group homomorphism
	\[
		\mathcal{N}\colon \Aut(X,C)\longrightarrow \Aut(\mathcal{N}_{C/X}\to C)
	\]
	via \(\sigma\mapsto \mathcal{N}_{\sigma}\) (cf.~\cite{meng2022jordan}*{Notation~2.2}).
	By~\cite{meng2022jordan}*{Lemma~2.3}, \(\ker\mathcal{N}\cap G=\{\id\}\).
	Moreover, by the proof of~\cite{meng2022jordan}*{Theorem~3.1},
	there is a monomorphism
	\[
		\Aut(\mathcal{N}_{C/X}\to C)\longrightarrow \Aut(\mathbb{P}_C(\mathcal{N}_{C/X}\oplus \mathcal{O})).
	\]
	So we may view \(G\) as a subgroup of \(\Aut(\mathbb{P}_C(\mathcal{N}_{C/X}\oplus \mathcal{O}))\).
	Note that \(C\) is projective.
	So \(\mathbb{P}_C(\mathcal{N}_{C/X}\oplus \mathcal{O})\) is also a projective manifold.
	Then the theorem follows from~\cref{thm:class_c}.
\end{proof}

\begin{corollary}\label{cor:alg_dim_zero_contains_curves}
	Let \(X\) be a compact complex surface with \(a(X)=0\).
	Assume that \(X\) contains at least one curve.
	Then \(\Aut(X)\) is T-Jordan.
\end{corollary}

\begin{proof}
	It follows from \cref{lem:cpt_surface_with_inv_curves},
	since \(X\) contains at most a finite number of (irreducible) curves
	(cf.~\cite{barth2004compact}*{IV.\ Theorem~8.2}).
\end{proof}

\begin{proposition}\label{prop:algebraic_dim_one_and_nonzero_euler_t_jordan}
	Let \(X\) be a compact complex surface with \(e(X)\neq 0\) and \(a(X)=1\).
	Then \(\Aut(X)\) is T-Jordan.
\end{proposition}

\begin{proof}
	Let \(\pi\colon X\longrightarrow Y\) be the algebraic reduction (cf.~\cref{sec:equivariant_fibrations}),
	so that \(Y\) is a smooth curve and \(\pi\) is an \(\Aut(X)\)-equivariant elliptic fibration.
	Since \(e(X)\neq 0\),
	the fibration \(\pi\) has at least one fibre \(X_y\) such that
	\(F=(X_y)_{\mathrm{red}}\) is not a smooth elliptic curve by Suzuki's formula (cf.~\cite{barth2004compact}*{\Romannum{3}, Proposition~11.4}).
	Note that \(\pi\) has only finitely many singular fibres,
	whose union is \(\Aut(X)\)-invariant.
	Then the assertion follows from \cref{lem:cpt_surface_with_inv_curves}.
\end{proof}

\section{Hopf Surfaces}\label{sec:hopf_surfaces}

In this section, we study the automorphism groups of Hopf surfaces.
By definition, a \emph{Hopf surface} \(X\) is a smooth compact complex surface with universal covering
being isomorphic to \(\mathbb{C}^2\setminus \{0\}\).
Then, \(X\) can be obtained from \(\mathbb{C}^2\setminus \{0\}\) as a quotient by a free action
of a discrete group \(\Gamma\simeq \pi_1(X)\).
If \(\pi_1(X)\simeq \mathbb{Z}\), we call such \(X\) a \emph{primary} Hopf surface.
In this case, after an appropriate choice of coordinates of \(\mathbb{C}^2\),
the generator of \(\Gamma\) has the form
\begin{equation}\label{eq:primary_hopf_surface}
	(z,w) \longmapsto (az+\lambda w^m, bw),
\end{equation}
where \(m \in \mathbb{Z}_{>0}\) and \(a,b,\lambda\) are complex constants subject to the restrictions
\[
	(a - b^m)\lambda = 0, \quad 0 < \abs{a} \leq \abs{b} < 1.
\]
Any Hopf surface is either primary or a finite \'etale quotient of a primary Hopf surface,
where the latter is called the \emph{secondary} Hopf surface
(cf.~\cite{kodaira1966structure}*{Theorem~30}).

By above, \(\Gamma\) contains \(\Lambda\simeq \mathbb{Z}\) as a subgroup of finite index.
In particular, \(\Gamma\) is virtually solvable.
After replacing \(\Lambda\) by a suitable subgroup \(\Lambda_0\simeq \mathbb{Z}\),
we may assume that \(\Lambda\) is a characteristic subgroup of \(\Gamma\)
(cf.~\cite{prokhorov2021automorphism}*{Lemma~2.10}),
with the generator of \(\Lambda\) having the form \cref{eq:primary_hopf_surface}.
Now there is a short exact sequence
\begin{equation}\label{eq:hopf_ses}
	1\longrightarrow \Gamma\longrightarrow \widetilde{\Aut}(X)\longrightarrow \Aut(X)\longrightarrow 1
\end{equation}
where \(\widetilde{\Aut}(X)\) acts on \(\mathbb{C}^2\setminus \{0\}\) biholomorphically.
It follows from the Hartogs extension theorem that
\(\widetilde{\Aut}(X)\) can be extended to
biholomorphic actions on \(\mathbb{C}^2\)
fixing the origin \(0\in \mathbb{C}\).
By assumption, \(\Lambda\lhd \widetilde{\Aut}(X)\) is a normal subgroup.

\begin{theorem}\label{thm:tits_alternative_hopf}
	Let \(X\) be a Hopf surface.
	Then the group \(\Aut(X)\) satisfies the Tits alternative.
\end{theorem}

\begin{proof}
	Assume first that \(X\) is a primary Hopf surface.
	Then we may identify \(\widetilde{\Aut}(X)\) with a subgroup of \(\GL_2(\mathbb{C})\times\mathbb{C}\)
	by \cite{namba1974automorphism}*{\S~2} and \cite{wehler1981versal}*{pp.~24}.
	Since \(\GL_2(\mathbb{C}) \times \mathbb{C}\) satisfies the Tits alternative,
	so do \(\widetilde{\Aut}(X)\) and \(\Aut(X)\) (cf.~\cref{lem:virtually_solvable_ses}~(4)).

	Now let \(X\) be a secondary Hopf surface.
	Then either \(\widetilde{\Aut}(X)\leq \GL_2(\mathbb{C})\),
	or \(\widetilde{\Aut}(X)\simeq \mathbb{C}\rtimes \mathbb{C}^{*}\) or
	\(\widetilde{\Aut}(X)\simeq \mathbb{C}\rtimes (\mathbb{C}^{*})^2\) (cf.~\cite{matumoto2000explicit}*{Theorem~1}).
	In the first case, \(\widetilde{\Aut}(X)\) and hence \(\Aut(X)\) satisfies the Tits alternative as \(\GL_2(\mathbb{C})\) does;
	in the latter two cases, \(\widetilde{\Aut}(X)\) and hence \(\Aut(X)\) are already solvable.
	In a word, \(\Aut(X)\) satisfies the Tits alternative.
\end{proof}

\begin{theorem}\label{thm:derived_length_hopf}
	Let \(X\) be a Hopf surface and \(G\leq \Aut(X)\) a virtually solvable subgroup.
	Then \(\ell_{\mathrm{vir}}(G)\leq 2\).
\end{theorem}

\begin{proof}
	Let \(\widetilde{G}\leq \widetilde{\Aut}(X)\) be the preimage of \(G\).
	Then \(\widetilde{G}\) is virtually solvable by \cref{lem:virtually_solvable_ses}~(3)
	as \(\Gamma\) is virtually solvable.
	Therefore, \(\ell_{\mathrm{vir}}(G)\leq \ell_{\mathrm{vir}}(\widetilde{G}) \leq \ell_{\mathrm{vir}}(\widetilde{\Aut}(X)) \leq 2\)
	by \cref{lem:derived_length_linear_gp} and proof of \cref{thm:tits_alternative_hopf}.
\end{proof}

The lemma below is a simple linear algebra, which is taken from \cite{prokhorov2021automorphism}*{Lemma~6.3}.

\begin{lemma}\label{lem:upper_triangular_centraliser}
	Let
	\[
		M=
		\begin{pmatrix}
			a & \lambda \\
			0 & b
		\end{pmatrix}
		\in \GL_2(\mathbb{C})
	\]
	be an upper triangular matrix,
	and \(Z\leq \GL_2(\mathbb{C})\) the centraliser of \(M\).
	Then the following assertions hold:
	\begin{enumerate}[label=(\roman*)]
		\item If \(a=b\) and \(\lambda=0\), then \(Z=\GL_2(\mathbb{C})\).
		\item If \(a\neq b\) and \(\lambda=0\), then \(Z\simeq (\mathbb{C}^{*})^2\).
		\item If \(a=b\) and \(\lambda\neq 0\), then \(Z\simeq \mathbb{C}^{*}\times\mathbb{C}^+\).
	\end{enumerate}
\end{lemma}

\begin{theorem}\label{thm:t_jordan_hopf}
	Let \(X\) be a Hopf surface.
	Then \(\Aut(X)\) is T-Jordan.
\end{theorem}

\begin{proof}
	The proof is adapted from \cite{prokhorov2021automorphism}*{Lemma~6.4}.

	Consider the short exact sequence \cref{eq:hopf_ses}.
	The image of the generator of \(\Lambda\) is mapped by the natural homomorphism
	\[
		\sigma\colon \widetilde{\Aut}(X)\longrightarrow \GL(T_{0,\mathbb{C}^2})\simeq \GL_2(\mathbb{C})
	\]
	to the matrix
	\[
		M=
		\begin{pmatrix}
			a & \lambda \delta_1^m \\
			0 & b
		\end{pmatrix}
	\]
	where \(\delta\) is the Kronecker symbol.

	Let \(G\leq \Aut(X)\) be a torsion subgroup,
	and \(\widetilde{G}\) its preimage in \(\widetilde{\Aut}(X)\).
	Thus, one has \(G\simeq \widetilde{G}/\Gamma\).
	By~\cref{lem:torsion_quotient_embedding}, \(\sigma|_{\widetilde{G}}\) is an embedding
	since \(\sigma_{\Lambda}\) is a group monomorphism and \(\widetilde{G}/\Lambda\) is torsion.
	Let \(\Omega\) be the normaliser of \(\sigma(\Lambda)\) in \(\GL_2(\mathbb{C})\).
	By construction \(\sigma(\widetilde{G})\) is contained in the normaliser of
	\(\sigma(\Gamma)\) in \(\GL_2(\mathbb{C})\),
	which in turn is contained in \(\Omega\)
	because \(\Lambda\) is a characteristic subgroup of \(\Gamma\).
	Hence, every torsion subgroup of \(\Aut(X)\) is contained in the group \(\Omega/\sigma(\Gamma)\).
	On the other hand,
	\(\Omega/\sigma(\Gamma)\) is a quotient of \(\Omega/\sigma(\Lambda)\) by a finite subgroup
	isomorphic to \(\sigma(\Gamma)/\sigma(\Lambda)\).
	Thus, by~\cref{lem:t_jordan_exact_sequence},
	it is sufficient to show that the group \(\Omega/\sigma(\Lambda)\) is T-Jordan.

	Since \(\sigma(\Lambda)\simeq \mathbb{Z}\),
	the group \(\Omega\) has a (normal) subgroup \(\Omega'\) of index at most \(2\)
	that coincides with the centraliser of the matrix \(M\).
	It remains to check that the group \(\Omega'/\sigma(\Lambda)\) is T-Jordan.
	If \(\lambda=0\) and \(a=b\),
	it follows from \cref{lem:upper_triangular_centraliser}(\romannum{1}) that \(\Omega'/\sigma(\Lambda)\)
	is a connected Lie group and hence T-Jordan by \cref{thm:connected_lie_group_t_jordan}.
	If either \(\lambda=0\) and \(a\neq b\),
	or \(\lambda\neq 0\) and \(m\geq 2\),
	then this follows from \cref{lem:upper_triangular_centraliser}(\romannum{2})
	that \(\Omega'/\sigma(\Lambda)\) is abelian.
	If \(\lambda\neq 0\) and \(m=1\),
	then this follows from \cref{lem:upper_triangular_centraliser}(\romannum{3})
	that \(\Omega'/\sigma(\Lambda)\) is abelian.
\end{proof}

\section{Inoue Surfaces}\label{sec:inoue_surfaces}

An \emph{Inoue surface} \(X\) is a compact complex surface obtained from \(W \coloneqq \mathbb{H}\times\mathbb{C}\)
as a quotient by an infinite discrete group, where \(\mathbb{H}\) is the upper half complex plane.
Inoue surfaces are minimal surfaces in class
\(\mathrm{VII}\), contain no curve, and have the following numerical invariants:
\[
	a(X) = 0,\quad b_1(X) = 1,\quad b_2(X) = 0.
\]
There are three families of Inoue surfaces: \(S_M\), \(S^{(+)}\), and \(S^{(-)}\)
(cf.~\cite{inoue1974surfaces}), and we will study their automorphisms separately.

Since every holomorphic map from \(\mathbb{C}\) to \(\mathbb{H}\) is constant,
any automorphism \(u\) of \(W\) has the form
\begin{equation}\label{eq:u}
	u(w,z)=(s(w),t(w,z))
\end{equation}
where
\[
	s(w)=\frac{aw+b}{cw+d}, \quad a,b,c,d \in \mathbb{R} \text{ and } ad-bc>0.
\]
We may expand \(t\) in the power series of \(w\) and \(z\) at \((\sqrt{-1},0)\):
\[
	t(w,z)=\sum_{p \geq 0} C_{p,0}(w-\sqrt{-1})^p + C_{p,1}(w-\sqrt{-1})^p z.
\]

\subsection{Type \texorpdfstring{\(S_M\)}{SM}}\label{sub:inoue_sm}

Let \(M = (m_{i,j})\in \SL_3(\mathbb{Z})\) be a matrix with eigenvalues
\(\alpha,\beta,\overline{\beta}\) such that \(\alpha>1\) and \(\beta\neq \overline{\beta}\).
Take \((a_1,a_2,a_3)^T\) to be a real eigenvector of \(M\) corresponding to \(\alpha\),
and \((b_1,b_2,b_3)^T\) an eigenvector corresponding to \(\beta\).
Let \(G_M\) be the group of automorphisms of \(W\) generated by
\begin{align*}
	g_0(w,z)   & = (\alpha w,\beta z),                    \\
	g_{i}(w,z) & = (w + a_{i}, z + b_{i}), \quad i=1,2,3,
\end{align*}
which satisfy these conditions
\begin{gather*}
	g_0 g_{i} g_0^{-1} = g_1^{m_{i,1}}g_2^{m_{i,2}}g_3^{m_{i,3}}, \\
	g_{i}g_{j} = g_{j}g_{i}, \quad i,j=1,2,3.
\end{gather*}
It can be shown that the action of \(G_M\) on \(W\) is free and properly discontinuous.
The quotient \(X\coloneqq W / G_M\) is an Inoue surface of type \(S_M\).
Thus, there is a short exact sequence of groups
\[
	1\longrightarrow G_M\longrightarrow \widetilde{\Aut}(X)\longrightarrow \Aut(X)\longrightarrow 1
\]
where \(\widetilde{\Aut}(X)\) acts biholomorphically on \(W\), which is the normaliser of \(G_M\) in \(\Aut(W)\).

We will study \(\widetilde{\Aut}(X)\) in detail.
Assume that \(u\in \widetilde{\Aut}(X)\).
Since \(G_M\lhd \widetilde{\Aut}(X)\), we have \(ugu^{-1}\in G_M\) for all \(g\in G_M\).
Indeed, we only need to verify this for the generators of \(G_M\).

By the commutative relations of \(G_M\), we may assume that
\begin{equation}\label{eq:inoue_sm_main}
	ug_{i}u^{-1}=g_1^{n_{i,1}}g_2^{n_{i,2}}g_3^{n_{i,3}}g_0^{k_i}, \quad 0 \leq i\leq 3
\end{equation}
for some \(k_i \in \mathbb{Z}\), \(\mathbf{n} = (n_{0,1}, n_{0,2}, n_{0,3}) \in \mathbb{Z}^3\)
and \(N\coloneqq (n_{i,j})_{1\leq i,j\leq 3} \in \Mat_{3\times 3}(\mathbb{Z})\).
Then
\begin{equation*}
	g_1^{n_{i,1}}g_2^{n_{i,2}}g_3^{n_{i,3}}g_0^{k_i} u(w,z)
	=(\alpha^{k_i} s(w)+\sum_{j=1}^{3}n_{i,j}a_{j},
	\beta^{k_i} t(w,z)+\sum_{j=1}^{3}n_{i,j}b_{j}).
\end{equation*}
Note that
\begin{align*}
	ug_0(w,z) & =u(\alpha w, \beta z)=(s(\alpha w),t(\alpha w,\beta z)), \\
	ug_i(w,z) & =u(w + a_i, z + b_i)=(s(w + a_i),t(z + b_i)).
\end{align*}
Then, \cref{eq:inoue_sm_main} implies
\begin{align}
	s(\alpha w)          & = \alpha^{k_0} s(w) + \sum_j n_{0,j} a_j, \label{eq:sm_g0_s} \\
	t(\alpha w, \beta z) & = \beta^{k_0} t(w,z)+\sum_j n_{0,j}b_j, \label{eq:sm_g0_t}   \\
	s(w + a_i)           & = \alpha^{k_i} s(w) + \sum_j n_{i,j} a_j, \label{eq:sm_gi_s} \\
	t(z + b_i)           & = \beta^{k_i} t(w,z)+\sum_j n_{i,j}b_j. \label{eq:sm_gi_t}
\end{align}

By \cref{eq:sm_g0_s}, we have
\begin{equation}\label{eq:sm_g0_s_expand}
	\begin{cases}
		ac(1 - \alpha^{k_0}) = c^2 \sum_j n_{0,j} a_j,                                            \\
		ad(\alpha - \alpha^{k_0}) + bc(1 - \alpha^{k_0 + 1}) = cd (1 + \alpha) \sum_j n_{0,j}a_j, \\
		bd(1 - \alpha^{k_0}) = d^2 \sum_j n_{0,j} a_j.
	\end{cases}
\end{equation}
If \(cd \neq 0\), then \(a(1 - \alpha^{k_0}) = c \sum_j n_{0,j} a_j\) and \(b(1 - \alpha^{k_0}) = d \sum_j n_{0,j} a_j\)
by the first and the third equalities of \cref{eq:sm_g0_s_expand}.
Using the fact that \(ad - bc > 0\),
the second equality of \cref{eq:sm_g0_s_expand} implies \(k_0 = 0\) and \(\sum_j n_{0,j} a_j = 0\),
which contradicts the middle equality above.
Therefore, either \(c = 0\) or \(d = 0\).

From \cref{eq:sm_gi_s} one deduces that
\begin{equation}\label{eq:sm_gi_s_expand}
	\begin{cases}
		ac(1 - \alpha^{k_i}) = c^2 \sum_j n_{i,j} a_j,                                                            \\
		(aca_i + ad + bc) (1 - \alpha^{k_i}) = (c^2 a_i + 2cd) \sum_j n_{i,j}a_j, \hspace{4em} \text{for all } i. \\
		(ad - bc\alpha^{k_i}) a_i + bd(1 - \alpha^{k_i}) = (cd a_i + d^2) \sum_j n_{i,j} a_j,
	\end{cases}
\end{equation}
If \(d = 0\), then \(bc\alpha^{k_i} a_i=0\) for all \(i\) by the last equality of \cref{eq:sm_gi_s_expand}.
Since \(bc \neq 0\) now and \(\alpha > 1\),
the third equality of \cref{eq:sm_gi_s_expand} implies \(a_i=0\) for all \(i\),
which contradicts the assumption that \((a_1,a_2,a_3)\) is an eigenvector of \(M\).
Hence, \(c=0\) and we may rewrite \(s\) as \(s(w) = aw + b\) with \(a>0\) and \(b \in \mathbb{R}\) (i.e., \(d = 1\)).
Then, by \cref{eq:sm_g0_s_expand}, \(k_0 = 1\) and \((1-\alpha)b = \sum_j n_{0,j}a_j\) and
\cref{eq:sm_gi_s_expand} gives \(k_i = 0\) and \(a a_i = \sum_j n_{i,j} a_j\).

Comparing the coefficients of \(z\) in \cref{eq:sm_g0_t}, we have
\begin{align}
	\sum_p C_{p,0} (\alpha w - \sqrt{-1})^p    & =
	\beta \sum_p C_{p,0} (w - \sqrt{-1})^p + \sum_j n_{0,j} b_j, \label{eq:sm_g0_t_constant}     \\
	\beta \sum_p C_{p,1}(\alpha w-\sqrt{-1})^p & =\beta \sum_p C_{p,1}(w-\sqrt{-1})^p. \nonumber
\end{align}
It follows that \(C_{p,0} = C_{p,1} = 0\) for all \(p\geq 1\).
Thus, we may rewrite \(t\) as \(t(w,z) = Az + B\) for some \(A\neq 0, B \in \mathbb{C}\).
Then \cref{eq:sm_g0_t_constant} becomes \((1-\beta)B = \sum_j n_{0,j}b_j\),
and \cref{eq:sm_gi_t} gives \(Ab_i = \sum_j n_{i,j}b_j\).

We conclude that \(u(w,z) = (aw+b,Az+B)\) where \(a,b,A,B\) satisfy
\[
	\mathbf{n} \cdot
	\begin{pmatrix}
		a_1 & b_1 \\
		a_2 & b_2 \\
		a_3 & b_3
	\end{pmatrix}
	=
	\begin{pmatrix}
		(1-\alpha)b & (1-\beta)B
	\end{pmatrix}
\]
and
\begin{equation}\label{eq:inoue_sm_n}
	N
	\begin{pmatrix}
		a_1 \\
		a_2 \\
		a_3
	\end{pmatrix}
	=a
	\begin{pmatrix}
		a_1 \\
		a_2 \\
		a_3
	\end{pmatrix},
	\quad N
	\begin{pmatrix}
		b_1 \\
		b_2 \\
		b_3
	\end{pmatrix}
	=A
	\begin{pmatrix}
		b_1 \\
		b_2 \\
		b_3
	\end{pmatrix}.
\end{equation}
It follows that \(M\) and \(N\) are simultaneously diagonalisable.
So \(M\) and \(N\) commute.
Moreover, the matrix associated to \(u^{-1}\) is \(N^{-1} \in \Mat_{3\times 3}(\mathbb{Z})\),
so \(\det N=\pm 1\) and \(N\in \GL_3(\mathbb{Z})\).

Consider the following subgroups of \(\widetilde{\Aut}(X)\):
\[
	K = \Big\{(w,z) \longmapsto
	\Big(w+\frac{1}{1-\alpha}\sum_{j=1}^{3}n_{i}a_{i},z+\frac{1}{1-\beta}\sum_{j=1}^{3}n_{i}b_{i}\Big)
	\mid n_i \in \mathbb{Z}\Big\} \simeq \mathbb{Z}^3
\]
and
\[
	\Gamma\coloneqq \{N\in \GL_3(\mathbb{Z}) \mid N \text{ and } M \text{ are simultaneously diagonalisable}\},
\]
which are abelian groups.
By the construction, \(\widetilde{\Aut}(X) = K \rtimes \Gamma\).
Note also that \(G_M = G_1 \rtimes G_0\) where
\(G_1 = \{g_1^{n_1}g_2^{n_2}g_3^{n_3} \mid n_i \in \mathbb{Z}\} \simeq \mathbb{Z}^3\)
and \(G_0 = \langle g_0 \rangle \simeq \mathbb{Z}\).
Now we have the following commutative diagram (by the snake lemma)
\begin{equation}\label{eq:sm_comm_diag}
	\xymatrix{
	{} & 1 \ar[d] & 1 \ar[d] & 1 \ar[d] & {} \\
	1 \ar[r] & G_1 \ar[r]^{\phi} \ar[d] & K \ar[r] \ar[d]^{\iota} & F \ar[r] \ar[d] & 1 \\
	1 \ar[r] & G_M \ar[r] \ar[d] & \widetilde{\Aut}(X) \ar[r] \ar[d] & \Aut(X) \ar[r] \ar[d] & 1 \\
	1 \ar[r] & G_0 \ar[r] \ar[d] & \Gamma \ar[r] \ar[d] & \Gamma/G_0 \ar[r] \ar[d] & 1 \\
	{} & 1 & 1 & 1 & {}
	}
\end{equation}
where \(\phi\) is given by
\[
	\phi(\mathbf{n})=\mathbf{n}\cdot (I-M) \quad \text{with } \mathbf{n}=(n_1,n_2,n_3).
\]
Since \(I - M\) is invertible, the group \(F\) being the cokernel of \(\phi\), is finite (and abelian).
Then, by the last column of the commutative diagram, \(\Aut(X)\) is metabelian.
By \cref{eq:sm_comm_diag,lem:virtually_solvable_ses} (1), we have:

\begin{theorem}\label{thm:tits_altervative_derived_length_inoue_sm}
	Let \(X\) be an Inoue surface of type \(S_M\).
	Then \(\Aut(X)\) is solvable.
	Moreover, for any \(G\leq \Aut(X)\), one has \(\ell_{\mathrm{vir}}(G)\leq 2\).
\end{theorem}

Next, we claim that any torsion subgroup of \(\Gamma/\mathbb{Z}\) is finite.
Let \(n_0\) be the largest integer such that,
for some \(n_0\)-th root of \(M\), \(M^{\frac{1}{n_0}}\in \GL_3(\mathbb{Z})\).
Let \(N\in \Gamma\) be an element such that its image in \(\Gamma/\mathbb{Z}\) is torsion.
Then there is a pair of coprime integers \((m,n)\) such that
\(N^n=M^m\) and hence \(M^\frac{m}{n}\in \GL_3(\mathbb{Z})\).
Since \(m\) and \(n\) are coprime, there are integers \(a\) and \(b\) such that \(am+bn=1\).
Then
\[
	\frac{1}{n}=a\frac{m}{n}+b.
\]
It follows that \(M^{\frac{1}{n}}\in \GL_3(\mathbb{Z})\) and hence \(n\leq n_0\).
Therefore,
\[
	\{(m,n)\mid M^{\frac{m}{n}}\in \GL_3(\mathbb{Z}), m,n \text{ coprime, } m<n\}
\]
is a finite set.
It follows that the set of torsion elements of \(\Gamma/\mathbb{Z}\) is finite (with the bound depending on \(X\)), which proves the claim.
This claim and \cref{eq:sm_comm_diag} imply:

\begin{theorem}\label{thm:t_jordan_inoue_sm}
	Let \(X\) be an Inoue surface of type \(S_M\).
	Then any torsion subgroup \(G\leq \Aut(X)\) is finite.
	In particular, \(\Aut(X)\) is T-Jordan.
\end{theorem}

\subsection{Types \texorpdfstring{\(S^{(+)}\)}{S+} and \texorpdfstring{\(S^{(-)}\)}{S-}}\label{sub:inoue_spm}

Since the constructions of Inoue surfaces of type \(S^{(+)}\) and \(S^{(-)}\) are almost parallel,
we will only focus on type \(S^{(+)}\) in this subsection.
The same argument works for Inoue surfaces of type \(S^{(-)}\).

Let \(M\in \SL_2(\mathbb{Z})\) be a matrix with two real eigenvalues \(\alpha\) and \(1/\alpha\) with \(\alpha>1\).
Let \((a_1, a_2)^T\) and \((b_1, b_2)^T\) be real eigenvectors of \(M\)
corresponding to \(\alpha\) and \(1/\alpha\), respectively,
and fix integers \(p_1,p_2,r\,(r\neq 0)\) and a complex number \(\tau\).
Define \((c_1, c_2)^T\) to be the solution of the following equation
\[
	(I-M)
	\begin{pmatrix}
		c_1 \\
		c_2
	\end{pmatrix}
	=
	\begin{pmatrix}
		e_1 \\
		e_2
	\end{pmatrix}
	+\frac{b_1a_2-b_2a_1}{r}
	\begin{pmatrix}
		p_1 \\
		p_2
	\end{pmatrix},
\]
where
\[
	e_{i}=\frac{1}{2}m_{i,1}(m_{i,1}-1)a_1b_1+\frac{1}{2}m_{i,2}(m_{i,2}-1)a_2b_2+m_{i,1}m_{i,2}b_1a_2, \quad i=1,2.
\]
Let \(G_M^{(+)}\) be the group of analytic automorphisms of \(W = \mathbb{H} \times \mathbb{C}\) generated by
\begin{align*}
	g_0\colon (w,z)   & \mapsto (\alpha w, z+\tau),                     \\
	g_{i}\colon (w,z) & \mapsto (w+a_{i}, z+b_{i}w+c_{i}), \quad i=1,2, \\
	g_3\colon (w,z)   & \mapsto \Big(w, z+\frac{b_1a_2-b_2a_1}{r}\Big).
\end{align*}
We have the following relations between these generators
\begin{gather*}
	g_3g_{i}=g_{i}g_3 \quad \text{for } i = 0,1,2, \quad g_1^{-1}g_2^{-1}g_1g_2=g_3^r, \\
	g_0g_jg_0^{-1}=g_1^{m_{j,1}}g_2^{m_{j,2}}g_3^{p_{j}} \quad \text{for }  j=1,2.
\end{gather*}
The action of \(G_M^{(+)}\) is free and properly discontinuous.
The quotient space \(X\coloneqq W/G_M^{(+)}\) is an Inoue surface of type \(S^{(+)}\).
Note that \(G_M^{(+)} \simeq H(r)\rtimes \mathbb{Z}\) as an abstract group,
where \(H(r) = \langle g_1, g_2, g_3 \mid g_3g_{i}=g_{i}g_3, g_1^{-1}g_2^{-1}g_1g_2=g_3^r \rangle\)
and \(\mathbb{Z}\) is generated by \(g_0\).
In fact, the centre \(Z(H(r)) \simeq \mathbb{Z}\) is generated by
\(g_3\)
and \(H(r)/Z(H(r)) \simeq \mathbb{Z}^2\).

Similarly, there is a short exact sequence of groups
\[
	1\longrightarrow G_M^{(+)}\longrightarrow \widetilde{\Aut}(X)\longrightarrow \Aut(X)\longrightarrow 1,
\]
where \(\widetilde{\Aut}(X)\leq \Aut(W)\) is the normaliser of \(G_M^{(+)}\).
Now suppose that \(u\in \widetilde{\Aut}(X)\) as in \cref{eq:u} and
\begin{equation}\label{eq:inoue_sp_main}
	ug_{i}u^{-1}=g_1^{n_{i,1}} g_2^{n_{i,2}} g_3^{l_{i}} g_0^{k_{i}}, \quad 0\leq i\leq 3,\; j=1,2
\end{equation}
for some \(l_{i},k_{i}\in \mathbb{Z}\),
\(\mathbf{n}_0 = (n_{0,1}, n_{0,2}), \mathbf{n}_3 = (n_{3,1}, n_{3,2}) \in \mathbb{Z}^2\)
and \(N = (n_{i,j})_{i,j=1,2} \in \Mat_{2\times 2}(\mathbb{Z})\).
Then,
\begin{align*}
	 & \phantom{\;=\;} g_1^{n_{i,1}} g_2^{n_{i,2}} g_3^{l_i} g_0^{k_i}u(w,z)                  \\
	 & = \Big(\alpha^{k_i}s(w)+\sum_j n_{i,j}a_j,t(w,z)+k_i \tau+l_i\frac{b_1a_2 - b_2a_1}{r}
	+ \Big(\sum_j n_{i,j}b_j\Big)\alpha^{k_i}s(w)+\sum_j n_{i,j}c_j+e_i(n)\Big),
\end{align*}
where
\[
	e_{i}(n)=\frac{1}{2}n_{i,1}(n_{i,1}-1)a_1b_1+\frac{1}{2}n_{i,2}(n_{i,2}-1)a_2b_2+n_{i,1}n_{i,2}a_2b_1.
\]
As in the previous subsection,
\cref{eq:inoue_sp_main} implies
\begin{align}
	s(\alpha w)                            & =\alpha^{k_0}s(w)+n_{0,1}a_1+n_{0,2}a_2, \label{eq:sp_g0_s}                                                                                   \\
	t(\alpha w,z+t)                        & =t(w,z) + k_0 \tau + l_0\frac{b_1a_2-b_2a_1}{r}+ \Big(\sum_j n_{0,j}b_j\Big)\alpha^{k_0} s(w)+\sum_j n_{0,j}c_j + e_0 (n), \label{eq:sp_g0_t} \\
	s(w+a_{i})                             & =\alpha^{k_i}s(w) + n_{i,1}a_1+n_{i,2}a_2, \label{eq:sp_gi_s}                                                                                 \\
	t(w + a_i, z + b_i w + c_i)            & =t(w,z) + k_i \tau + l_i\frac{b_1a_2-b_2a_1}{r}+ \Big(\sum_j n_{i,j}b_j\Big)\alpha^{k_i} s(w)+\sum_j n_{i,j}c_j+e_i (n), \label{eq:sp_gi_t}   \\
	s(w)                                   & =\alpha^{k_3}s(w) + n_{3,1}a_1+n_{3,2}a_2, \label{eq:sp_g3_s}                                                                                 \\
	t\Big(w,z+\frac{b_1a_2-b_2a_1}{r}\Big) & =t(w,z)+ k_3 \tau + l_3\frac{b_1a_2-b_2a_1}{r}+ \Big(\sum_j n_{3,j}b_j\Big)\alpha^{k_3} s(w)+\sum_j n_{3,j}c_j + e_3(n). \label{eq:sp_g3_t}
\end{align}

By \cref{eq:sp_g0_s}, using the assumption that \(\alpha>1\), we have
\begin{equation}\label{eq:sp_g0_s_expand}
	\begin{cases}
		ac (1-\alpha^{k_0}) = c^2 (n_{0,1}a_1+n_{0,2}a_2),                                       \\
		ad (\alpha-\alpha^{k_0}) + bc (1-\alpha^{k_0+1}) = cd (1+\alpha)(n_{0,1}a_1+n_{0,2}a_2), \\
		bd (1-\alpha^{k_0}) = d^2 (n_{0,1}a_1+n_{0,2}a_2).
	\end{cases}
\end{equation}
If \(c\neq 0\), then \(k_0=-1\) by the first two equalities of \cref{eq:sp_g0_s_expand} since \(ad-bc>0\).
Similarly, if \(d\neq 0\), then  one has \(k_0=1\) by the last two equalities of \cref{eq:sp_g0_s_expand}.
Therefore, either \(c=0\) and \(k_0=1\), or \(d=0\) and \(k_0=-1\).

By \cref{eq:sp_gi_s}, we have
\begin{equation}\label{eq:sp_gi_s_expand}
	\begin{cases}
		ac (1 - \alpha^{k_{i}}) = c^2(n_{i,1}a_1+n_{i,2}a_2),                              \\
		(aca_{i} + bc + ad) (1 - \alpha^{k_i}) = (c^2a_{i} + 2cd) (n_{i,1}a_1+n_{i,2}a_2), \\
		ada_{i} + bd (1 - \alpha^{k_{i}}) =
		\alpha^{k_{i}}bca_{i} + (cda_{i} + d^2) (n_{i,1}a_1+n_{i,2}a_2).
	\end{cases}
\end{equation}
If \(d=0\), we have \(\alpha^{k_{i}}bca_{i}=0\) for \(i=1,2\) by the last equality of \cref{eq:sp_gi_s_expand}.
This implies that \(a_{i}=0\; (i=1,2)\) since \(\alpha>1\) and \(ad-bc>0\),
which contradicts the fact that \((a_1,a_2)^T\) is an eigenvector of \(M\).
So, \(c=0\) and \(k_0=1\).
Also, the second equality of \cref{eq:sp_gi_s_expand}, which becomes \(ad=\alpha^{k_{i}}ad\), implies that \(k_{i}=0\) for \(i=1,2\).
Now we may rewrite \(s\) as \(s(w)=aw+b\) with \(a>0\) and \(b \in \mathbb{R}\).
Then \cref{eq:sp_g0_s_expand,eq:sp_gi_s_expand} imply that
\begin{equation}\label{eq:inoue_sp_relations_1}
	\begin{cases}
		(1-\alpha)b = n_{0,1}a_1+n_{0,2}a_2, \\
		aa_{i} =n_{i,1}a_1+n_{i,2}a_2, \quad i = 1, 2.
	\end{cases}
\end{equation}

Comparing the coefficients of \(z\) of \cref{eq:sp_g0_t}, we have
\[
	\sum_p C_{p,1}(\alpha w-\sqrt{-1})^p = \sum_p C_{p,1}(w-\sqrt{-1})^p,
\]
which implies that \(C_{p,1}=0\) for all \(p\geq 1\).
For the \(z\)-constant part of \cref{eq:sp_g0_t}, one has the following equality
\begin{align*}
	 & \phantom{\;=\;} \sum_p C_{p,0}(\alpha w-\sqrt{-1})^p + C_{0,1} \tau                                                                \\
	 & =\sum_p C_{p,0}(w-\sqrt{-1})^p + (n_{0,1}b_1+n_{0,2}b_2)\alpha aw+ \tau +l_0 \frac{b_1a_2-b_2a_1}{r}+n_{0,1}c_1+n_{0,1}c_2+e_0(n).
\end{align*}
Similarly, we conclude that \(C_{p,0}=0\) for \(p\geq 2\).
Hence, we may rewrite \(t\) as \(t(w,z)=Aw+B+Cz\) for some \(A,B,C\in \mathbb{C}\) with \(C\neq 0\).
Now comparing the coefficients of \(w\) of \cref{eq:sp_g0_t,eq:sp_gi_t},
we get
\begin{equation}\label{eq:inoue_sp_relations_2}
	\begin{cases}
		\Big(1-\frac{1}{\alpha}\Big) \frac{A}{a}  = n_{0,1}b_1+n_{0,2}b_2,                                                                   \\
		(C-1) \tau                                =l_0 \frac{b_1a_2-b_2a_1}{r}+(n_{0,1}b_1+n_{0,2}b_2)\alpha b+n_{0,1}c_1+n_{0,2}c_2+e_0(n), \\
		Cb_{i}                                   = a(n_{i,1}b_1+n_{i,2}b_2),                                                                 \\
		Aa_{i}+Cc_{i}                             =l_{i}\frac{b_1a_2-b_2a_1}{r}+(n_{i,1}b_1+n_{i,2}b_2)b+n_{i,1}c_1+n_{i,2}c_2+e_{i}(n).
	\end{cases}
\end{equation}
Finally, \cref{eq:sp_g3_s,eq:sp_g3_t} simply give that \(k_3 = 0\), \(l_3=C\) and \(n_{3,1}=n_{3,2}=0\).

Therefore, \(N\) has two eigenvalues \(a\) and \(C/a\)
with eigenvectors \((a_1,a_2)^T\) and \((b_1,b_2)^T\), respectively.
It is not hard to see that the matrix associated to \(u^{-1}\) is \(N^{-1}\).
Then \(\det N=\pm 1\) and hence \(N\in \GL_2(\mathbb{Z})\).
Moreover, by \cref{eq:inoue_sp_relations_1,eq:inoue_sp_relations_2}, \(u\) (if exists) is determined by \(\mathbf{n}_0\), \(N\) and \(B\).

Let
\[
	\Gamma\coloneqq \{N\in \GL_2(\mathbb{Z}) \mid N \text{ and } M \text{ are simultaneously diagonalisable}\},
\]
which is an abelian group,
and let \(\tau\colon \widetilde{\Aut}(X)\longrightarrow \Gamma\) be the homomorphism
(not necessarily surjective) defined by \(u\mapsto N\),
where \(N\) is the matrix associated with \(u\) as constructed above.
Let \(K\) be the kernel of this homomorphism, with image \(\Gamma'\leq \Gamma\).
It is clear that any automorphism in \(K\) has the form
\[
	(w,z)\longmapsto (w+b,z+Aw+B)
\]
for some \(b\in \mathbb{R}\) and \(A,B\in \mathbb{C}\) satisfying \cref{eq:inoue_sp_relations_1,eq:inoue_sp_relations_2}.
An explicit calculation gives that the centre \(Z(K)\simeq \mathbb{C}\), which is generated by
\[
	(w,z)\longmapsto (w,z+B).
\]
Further, \(K/Z(K)\simeq \mathbb{Z} \rtimes \mathbb{Z}\) is generated by the images of
\[
	(w,z)\longmapsto (w+b,z) \quad \text{and} \quad (w,z)\longmapsto (w,z+Aw).
\]

Combining these, the snake lemma shows the diagrams below
are commutative with exact rows and columns.
\[
	\xymatrix@C=1pc{
	{} & 1 \ar[d] & 1 \ar[d] & 1 \ar[d] & {} \\
	1 \ar[r] & H(r) \ar[r] \ar[d] & K \ar[r] \ar[d] & F \ar[r] \ar[d] & 1 \\
	1 \ar[r] & G_M^{(+)} \ar[r] \ar[d] & \widetilde{\Aut}(X) \ar[r] \ar[d]^{\tau} & \Aut(X) \ar[r] \ar[d] & 1 \\
	1 \ar[r] & \mathbb{Z} \ar[r] \ar[d] & \Gamma' \ar[r] \ar[d] & \Gamma'/\mathbb{Z} \ar[r] \ar[d] & 1 \\
	{} & 1 & 1 & 1 & {}
	}
	\hspace{3em}
	\xymatrix@C=1pc{
	{} & 1 \ar[d] & 1 \ar[d] & 1 \ar[d] & {} \\
	1 \ar[r] & Z(H(r)) \ar[r] \ar[d] & H(r) \ar[r] \ar[d] & H(r)/Z(H(r)) \ar[r] \ar[d] & 1 \\
	1 \ar[r] & Z(K) \ar[r] \ar[d] & K \ar[r] \ar[d] &K/Z(K) \ar[r] \ar[d] & 1 \\
	1 \ar[r] & F' \ar[r] \ar[d] & F \ar[r] \ar[d] &F'' \ar[r] \ar[d] & 1 \\
	{} & 1 & 1 & 1 & {}
	}
\]
Note that \(F'\) is abelian and \(F''\) is finite and metabelian.
Consequently, \(F\), as the quotient of \(K\) by \(H(r)\),
is virtually abelian and abelian-by-metabelian.

By the above diagrams and \cref{lem:virtually_solvable_ses} (1):
\begin{theorem}\label{thm:tits_altervative_derived_length_inoue_spm}
	Let \(X\) be an Inoue surface of type \(S^{(+)}\) or \(S^{(-)}\).
	Then \(\Aut(X)\) is solvable.
	Moreover, for any subgroup \(G\leq \Aut(X)\), one has \(\ell_{\mathrm{vir}}(G)\leq 4\).
\end{theorem}

Similar to the argument before \cref{thm:t_jordan_inoue_sm}, any torsion subgroup of \(\Gamma'/\mathbb{Z}\) is finite.
It follows from the diagram above that each torsion subgroup of \(\Aut(X)\)
has an abelian subgroup with bounded index depending on \(X\).

\begin{theorem}\label{thm:t_jordan_inoue_spm}
	Let \(X\) be an Inoue surface of type \(S^{(+)}\) or \(S^{(-)}\).
	Then \(\Aut(X)\) is T-Jordan.
\end{theorem}


\section{Class \texorpdfstring{\Romannum{7}}{VII} Surfaces with \texorpdfstring{\(b_2 > 0\)}{b2>0}}\label{sec:class_vii_surfaces_with_positive_b2}

We will prove \cref{prop:main_class_vii_no_curve} in this section.

\begin{proof}[Proof of \cref{prop:main_class_vii_no_curve}]
	We may assume that \(G\) is an infinite group.
	By \cref{cor:alg_dim_zero_contains_curves}, we only need to consider such an \(X\)
	that does not have any curves.
	Consider the action of \(\Aut(X)\) on \(H^*(X,\mathbb{Q})\),
	and let \(\Aut^*(X)\) be the kernel of \(\Aut(X)\longrightarrow \GL(H^*(X,\mathbb{Q}))\).
	Note that the image of \(G\) in \(\GL(H^*(X,\mathbb{Q}))\) is finite (cf.~\cite{jia2022equivariant}*{Lemma~5.2}).
	By passing to a finite index subgroup, we may assume that \(G\leq \Aut^*(X)\).

	Pick \(\id \neq g \in G\), and let \(G'\) be the centraliser of \(\langle g\rangle\) in \(G\).
	Since \(g\) has finite order, \([G:G']\) is finite.
	Replacing \(G\) by the finite-index subgroup \(G'\), we may assume that \(g \in Z(G)\),
	the centre of \(G\).
	Note that the fixed point set \(\Fix(g)\) of \(g\) is finite with cardinality \(|\Fix(g)| = b_2(X)\)
	(cf.~\cite{prokhorov2020automorphism}*{Lemma~5.4}).
	Consider the action of \(G\) on the finite set \(\Fix(g)\);
	noting that \(g\) is in the centre of \(G\).
	Replacing \(G\) by a subgroup of index \(\leq b_2(X)!\),
	we may assume that \(G\) fixes some point \(x \in \Fix(g)\).
	Applying \cref{lem:fix_point_torsion_embedding},
	the torsion group \(G\) is embedded into \(\GL_2(\mathbb{C})\).
	By \cref{thm:connected_lie_group_t_jordan}, we have \(G\) is virtually abelian.
\end{proof}

\section{Compact K\"ahler Surfaces}\label{sec:compact_kahler_surfaces}

In this section, we deal with smooth compact surfaces which are K\"ahler.
The idea of this section comes from De-Qi Zhang.

Let \(X\) be a compact K\"ahler manifold.
For a subgroup \(G\) of \(\Aut(X)\), define \(G^0\coloneqq G \cap \Aut_0(X)\) and
denote by \(N(G)\) the set of all elements in \(G\) with zero entropy.

\begin{lemma}\label{lem:null_entropy}
	Let \(X\) be a smooth compact K\"ahler surface,
	and \(G \leq \Aut(X)\) a virtually solvable subgroup.
	Then \(\ell_{\mathrm{vir}}(G) \leq 1 +  \ell_{\mathrm{vir}}(N(G))\).
\end{lemma}

\begin{proof}
	By taking a finite-index subgroup,
	we may assume that \(G\) is a solvable,
	\(N(G)\) is a normal subgroup of \(G\)
	and \(G/N(G) \simeq \mathbb{Z}^r\)
	for some \(r = 0,1\) (cf.~\cite{campana2013automorphism}*{Theorem~1.5}).
	Take a finite-index subgroup \(N' \leq N(G)\) such that \(\ell(N') = \ell_{\mathrm{vir}}(N(G))\)
	and \(\ell(N'/N'^0) \leq 1\) as in \cite{dinh2022zero}*{Theorem~1.2 and Proposition~2.6}.

	Note that \(N(G)|_{H^2(X,\mathbb{C})}\) is finitely generated (cf.~\cite{campana2013automorphism}*{Theorem~2.2(1)}).
	Since \(N'|_{H^2(X,\mathbb{C})}\) has finite index in \(N(G)|_{H^2(X,\mathbb{C})}\),
	after replacing \(N'\) by a finite-index subgroup,
	we may assume that \(N'|_{H^2(X,\mathbb{C})}\) is characteristic in \(N(G)|_{H^2(X,\mathbb{C})}\) by \cref{lem:finite_index_normal}.
	In particular, \(N'|_{H^2(X,\mathbb{C})} \lhd G|_{H^2(X,\mathbb{C})}\).

	Let \(N \lhd N(G)\) be the preimage of \(N'|_{H^2(X,\mathbb{C})}\).
	Then \(N \lhd G\), \(N(G)/N\) is finite,
	\(\ell(N) = \ell_{\mathrm{vir}}(N(G))\) and \(\ell(N/N^0) \leq 1\).
	Since \((G/N)/(N(G)/N) \simeq \mathbb{Z}^r\),
	replacing \(G\) by a finite-index subgroup, we may assume \(G/N\simeq \mathbb{Z}^r\)
	(cf.~\cite{zhang2013algebraic}*{Lemma~2.4}).
	It follows that
	\[
		\ell_{\mathrm{vir}}(G) \leq \ell(G) \leq \ell(G/N) + \ell(N) \leq 1 + \ell_{\mathrm{vir}}(N(G)). \qedhere
	\]
\end{proof}

Indeed, it follows from the above proof that
\[
	\ell_{\mathrm{vir}}(G) \leq 1 + \ell(N) \leq 1 + \ell(N/N^0) + \ell(N^0) \leq 2 +  \ell(N^0).
\]
When \(\kappa(X) \geq 0\), \(\Aut_0(X)\) is a complex torus (cf.~\cite{fujiki1978automorphism}*{Corollary~5.11}).
In particular, \(N^0 = N \cap \Aut_0(X)\) is abelian and hence of derived length (at most) \(1\).
Thus, we have

\begin{proposition}\label{prop:derived_length_kahler_pos}
	Let \(X\) be a smooth compact K\"ahler surface with Kodaira dimension \(\kappa(X) \geq 0\).
	Let \(G \leq \Aut(X)\) be virtually solvable.
	Then \(\ell_{\mathrm{vir}}(G) \leq 3\).
\end{proposition}

Suppose that now \(X\) is a compact K\"ahler surface with \(\kappa(X) = -\infty\).
Then \(X\) is projective.

\begin{proposition}\label{prop:derived_length_unipotent}
	Let \(X\) be a smooth projective surface.
	Suppose that \(G \leq \Aut(X)\) is virtually solvable,
	and \(U \leq \Aut_0(X)\) is non-trivial unipotent with \(U \lhd G\).
	Then the derived length \(\ell(U) \leq 2\) and
	after replacing by a finite-index subgroup,
	we have \(G = G^0\).
\end{proposition}

\begin{proof}
	Replacing by closure or finite-index subgroup,
	we may assume that \(G\leq \Aut(X)\) is closed and solvable,
	\(G|_{\NS(X)_{\mathbb{C}}}\) is connected and closed, and \(U\leq \Aut_0(X)\) is closed and unipotent
	(cf.~\cite{oguiso2006tits}*{Proof of Lemma~2.1(2)}).

	By \cref{lem:one_dim_normal_subgroup},
	after replacing \(G\) by a finite-index subgroup,
	there is a subgroup \(Z_1\simeq \mathbb{G}_a \leq U\) which is normal in \(G\).
	Let \(X \dashrightarrow Y = X/Z_1\) be the quotient map to the cycle space
	(cf.~\cite{fujiki1978automorphism}*{\S~4}).
	Then \(G\) acts biregularly on \(Y\) and the map is \(G\)-equivariant.
	Replacing the map by the graph and taking equivariant resolutions,
	we may assume that both \(X\) and \(Y\) are smooth,
	and the map \(X\longrightarrow Y\) is a morphism,
	a \(\mathbb{P}^1\)-fibration,
	whose general fibre is a curve with dense orbit \(Z_1\cdot x \simeq \mathbb{A}^1\)
	and whose section at infinity is fixed by \(G\).
	Consider the exact sequence
	\[
		1\longrightarrow K\longrightarrow U\longrightarrow U|_Y\longrightarrow 1.
	\]
	Then both \(K\) and \(U|_Y\) are unipotent.
	Now the restriction to generic fibre \(K \longrightarrow K|_{X_{\overline{k(Y)}}}\) is injective
	with image contained in a closed unipotent group acting faithfully
	on \(X_{\overline{k(Y)}} = \mathbb{P}^1_{\overline{k(Y)}}\),
	and hence is abelian.
	It is clear that either \(U|_Y = \{1\}\),
	or \(U|_Y \simeq \mathbb{G}_a\) and \(Y = \mathbb{P}^1\).
	Hence \(\ell(U) \leq \ell(U|_Y) + 1 \leq 2\).

	Note that \(G\) fixes a big class \(H_Y\) on the curve \(Y\).
	Take the class \(H\) to be the sum of the pullback of \(H_Y\) and the
	section at infinity of the \(\mathbb{P}^1\)-fibration \(X\longrightarrow Y\).
	Then \(G\) fixes the big class \(H\) on \(X\) and hence replaced by a finite-index subgroup,
	\(G\) is contained in \(\Aut_0(X)\)
	(cf.~\cite{fujiki1978automorphism}*{Theorem~4.8}, \cite{lieberman1978compactness}*{Proposition~2.2} and \cite{dinh2015compact}*{Corollary~2.2}).
	Thus \(G = G^0\).
\end{proof}

\begin{proposition}\label{prop:derived_length_proj}
	Let \(G\) be a smooth projective surface.
	Suppose that \(G \leq \Aut(X)\)	is virtually solvable.
	Then \(\ell_{\mathrm{vir}}(G) \leq 4\).
\end{proposition}

\begin{proof}
	Replacing by finite-index subgroup,
	we may assume that \(G\) is solvable.

	First assume that \(G = G^0\) (and connected).
	Let \(U\) be the unipotent radical of (the linear part of) \(G = G^0\) so that
	\(G/U\) is a semi-torus and hence commutative (cf.~\cite{iitaka1976logarithmic}*{Lemma~4}).
	Thus, \(\ell_{\mathrm{vir}}(G) \leq 1 + \ell(U) \leq 3\) by \cref{prop:derived_length_unipotent}.

	Now we may assume that \(G \neq G^0\) even after replacing \(G\) by a finite-index subgroup.
	Let \(G^0_0 \lhd G^0\) be the identity component,
	Then by \cref{prop:derived_length_unipotent}, the unipotent radical of (the linear part of) \(G^0_0\)
	is trivial and hence \(G^0_0\) itself is a semi-torus and hence commutative.

	If \(G\) is of zero entropy, by \cite{dinh2022zero}*{Theorem~1.2} one has \(\ell(G/G^0) \leq 1\) after
	replacing \(G\) by a finite-index subgroup.
	Consider the short exact sequence
	\[
		1\longrightarrow G^0/G^0_0\longrightarrow G/G^0_0\longrightarrow G/G^0\longrightarrow 1.
	\]
	Since the first term is a finite group, its centraliser in \(G/G^0_0\) is of finite index
	(say of index \(1\), after replacing \(G\) by a finite index subgroup).
	Thus,
	\begin{gather*}
		\ell(G/G^0_0)          \leq 1 + \ell(G/G^0) = 2,                             \\
		\ell_{\mathrm{vir}}(G) \leq \ell(G) \leq \ell(G/G^0_0) + \ell(G^0_0) \leq 3.
	\end{gather*}

	If \(G\) is not of zero entropy,
	using \cref{lem:null_entropy} and the argument above,
	\(\ell_{\mathrm{vir}}(G) \leq 1 + \ell_{\mathrm{vir}}(N(G)) \leq 1 + 3 = 4\).
\end{proof}

\section{Summary}\label{sec:summary}

In the K\"ahler setting, \cref{prop:main_t_jordan} follows from \cref{thm:class_c};
\Cref{thm:main_tits_alternative} has been proved in \cite{campana2013automorphism}*{Theorem~1.5};
\Cref{thm:main_derived_length} follows from \cref{prop:derived_length_kahler_pos,prop:derived_length_proj}.

Let \(X\) be a smooth compact complex surface, which is not K\"ahler.
In particular, \(X\) is not rational or ruled.
Thus, there is a unique minimal surface \(X'\) bimeromorphic to \(X\) such that
\[
	\Aut(X) \subseteq \Bim(X) \simeq \Bim(X') = \Aut(X').
\]
See \cite{prokhorov2021automorphism}*{Proposition~3.5}.
It suffices for us to prove \cref{prop:main_t_jordan,thm:main_tits_alternative,thm:main_derived_length}
for minimal surfaces (in the non-K\"ahler setting).
Note that non-K\"ahler minimal surfaces are surfaces of class \Romannum{7},
(primary or secondary) Kodaira surfaces and some properly elliptic surfaces
(cf.~\cite{barth2004compact}*{IV.~Theorem~3.1 and VI.~Theorem~1.1}).
Since every minimal surface of class \Romannum{7} with vanishing \(b_2\)
is either a Hopf surface or an Inoue surface (cf.~\cite{bogomolov1976classification})
and
minimal surface of class \Romannum{7} with algebraic dimension \(1\) is a Hopf surface
(cf.~\cite{barth2004compact}*{V.~Theorem~18.6}),
minimal surfaces of class \Romannum{7} not in \(\Xi\) are exactly Hopf surfaces or Inoue surfaces.
\begin{table}[htbp]
	\caption{non-K\"ahler minimal smooth compact complex surfaces}
	\begin{tabular}{lccccc}
		\hline
		class of the surface \(X\)     & \(\kappa(X)\) & \(a(X)\) & \(b_1(X)\) & \(b_2(X)\) & \(e(X)\)   \\ \hline
		surfaces of class \Romannum{7} & \(-\infty\)   & \(0,1\)  & \(1\)      & \(\geq 0\) & \(\geq 0\) \\
		primary Kodaira surfaces       & \(0\)         & \(1\)    & \(3\)      & \(4\)      & \(0\)      \\
		secondary Kodaira surfaces     & \(0\)         & \(1\)    & \(1\)      & \(0\)      & \(0\)      \\
		properly elliptic surfaces     & \(1\)         & \(1\)    &            &            & \(\geq 0\) \\ \hline
	\end{tabular}
\end{table}

Then \Cref{prop:main_t_jordan} follows from \cref{thm:elliptic_fibration_all,cor:alg_dim_zero_contains_curves,thm:t_jordan_hopf,thm:t_jordan_inoue_sm,thm:t_jordan_inoue_spm};
\Cref{thm:main_tits_alternative} follows from
\cref{thm:inoue_hirzebruch,thm:elliptic_fibration_all,thm:enoki,thm:tits_alternative_hopf,thm:tits_altervative_derived_length_inoue_sm,thm:tits_altervative_derived_length_inoue_spm};
and \cref{thm:main_derived_length} follows from
\cref{thm:inoue_hirzebruch,thm:elliptic_fibration_all,thm:enoki,thm:derived_length_hopf,thm:tits_altervative_derived_length_inoue_sm,thm:tits_altervative_derived_length_inoue_spm}.
Finally, \cref{prop:main_class_vii_no_curve}
has been proved in \cref{sec:class_vii_surfaces_with_positive_b2}.


\end{document}